\documentclass[11pt]{article}

\usepackage{amsmath}
\usepackage{amsthm,amssymb,amsfonts}
\usepackage{natbib,abbrevs}

\def\top{\intercal}

\def\eq#1{(\ref{eq-#1})}
\def\be{\begin{equation}}
\def\ee{\end{equation}}

\def\e{\varepsilon}

\def\y{\mathbf{y}}

\def\C{\mathbf{C}}
\def\M{\mathbf{M}}
\def\N{\mathbf{N}}

\def\I{\mathbf{I}}
\def\O{\mathbf{0}}
\def\p{{\boldsymbol{\pi}}}

\def\simplex{\Delta}
\def\A{\mathbf{A}}

\def\c{\mathbf{c}}
\def\G{\mathcal{G}}
\def\X{\mathbf{X}}
\def\Y{\mathbf{Y}}
\def\M{\mathbf{M}}
\def\B{\mathbf{B}_\p}
\def\e{\mathbf{e}}

\newtheorem{theorem}{Theorem}

\newtheorem{defn}{Problem}

\newname\stableset{{\rm \textsc{independent-set}}}
\newname\mmatrix{{\rm \textsc{polytope-M-matrix}}}
\newname\posdet{{\rm \textsc{positive-determinant}}}
\newname\minradius{{\rm \textsc{minimize-radius}}}
\newname\polytopicinstability{{\rm \textsc{polytopic-instability}}}

\begin{document} 

\title{\vspace*{-50pt}NP-hardness of polytope M-matrix testing \\ and related problems}

\author{%
Nikos Vlassis\\
~\\
{\small Luxembourg Centre for Systems Biomedicine, University of Luxembourg} \\
{\small \texttt{nikos.vlassis@uni.lu}}}

\maketitle

\begin{abstract}
In this note we prove NP-hardness of the following problem: Given a set of matrices, is there a convex combination of those that is a nonsingular M-matrix? Via known characterizations of M-matrices, our result establishes NP-hardness of several fundamental problems in systems analysis and control, such as testing the instability of an uncertain dynamical system, and minimizing the spectral radius of an affine matrix function. 
\end{abstract}

\section{Introduction}

Several problems in the biological, physical, social sciences and engineering, can be modelled via matrices that possess special structure. A commonly encountered case involves the use of nonsingular M-matrices, which are positive stable matrices with all their off-diagonal entries nonpositive. M-matrices enjoy a remarkably rich theory, with connections to the Perron-Frobenius theory of nonnegative matrices, and with many applications in different scientific fields. The books of \cite{Horn91} and \cite{Berman94} provide  extensive treatments of the subject. 

Negated M-matrices constitute a special case of Hurwitz matrices that play a pivotal role in the stability of dynamical systems \citep{Khalil02}. A fundamental problem in this area is the analysis of the stability of a system under model uncertainty, with important applications in systems engineering and biology \citep{Doyle11}. Often this uncertainty manifests as a polytope of system matrices, whose overall stability or instability needs to be established \citep{Chesi10}. In this context, it is of interest to know what is the computational `effort' that it takes to test the (in)stability of a given uncertain system. 

In this note we prove NP-hardness of the problem of testing whether there exists a convex combination of a given set of matrices that is a nonsingular M-matrix. Using known characterizations of M-matrices, our result establishes NP-hardness of several other problems with practical significance, such as the problem of testing the instability of a linear time-invariant system with polytopic parameter uncertainty, and the problem of minimizing the spectral radius of an affine matrix function. NP-hardness of the latter problem was independently established recently \citep{Fercoq11}. 

\section{The \mmatrix problem}

A real matrix whose off-diagonal entries are all nonpositive is called a Z-matrix. If a Z-matrix is positive stable, that is, all its eigenvalues have positive real part, then it is called a nonsingular M-matrix \citep[chapter 6]{Berman94}. 

In this section we prove NP-hardness of the following decision problem:

\begin{defn}[The \mmatrix problem]
Given a finite set of real matrices, is there a convex combination of those that is a nonsingular M-matrix? 
\end{defn}

\begin{theorem}
The \mmatrix problem is NP-hard.
\label{np}
\end{theorem}

\begin{proof}
We establish a polynomial-time reduction from the \stableset problem. This problem asks, for a given undirected graph $\G=(V,E)$ and a positive integer $j \leq |V|$, whether $\G$ contains an independent set $V'$ (a set of pairwise non-adjacent vertices) with size $|V'| \geq j$. This problem is NP-complete~\citep{Garey79}. 

An instance of \mmatrix takes as input $k$ real $n \times n$ matrices $\A_i$, with $i=1,\ldots,k$, and asks whether there exists a nonnegative vector $\p=(\pi_1,\ldots,\pi_k)^\top$ with $\sum_{i=1}^k \pi_i = 1$, such that the matrix $\B = \sum_{i=1}^k \pi_i \A_i$ is a nonsingular M-matrix. We will prove NP-hardness of the problem for the special case in which $k=n$, that is, when the number of input matrices equals their dimension. For any input instance of \stableset, our reduction will construct an instance of \mmatrix, such that a polynomial-time algorithm for deciding the latter would imply a polynomial-time for deciding the former.

Let $(\G,j)$ be an instance of \stableset, and let $\C$ be the $n \times n$ adjacency matrix of the graph $\G$.
The matrix $\C$ is a symmetric zero-one matrix with all zeros in the main diagonal. Let $\c_i$ denote the $i$'th column of $\C$, and let $\e_i$ denote the length-$n$ vector with 1 in the $i$'th entry and all other entries zero. The reduction constructs $n$ block matrices $\A_i$, for $i=1,\ldots,n$, of size $(n+1) \times (n+1)$:
 \be
 \A_i = 
 \begin{bmatrix} 
  \I & -(\e_i+\c_i) \\
  -\e_i^\top  & \frac{1}{j}   
 \end{bmatrix},
 \label{eq-A}
\ee
where $\I$ is the $n \times n$ identity matrix. Note that each $\A_i$ is a Z-matrix.
 
Suppose there exists a polynomial-time algorithm that allows us to find a vector $\p=(\pi_1,\ldots,\pi_n)^\top$ in the probability simplex $\simplex$, for which the matrix $\B = \sum_{i=1}^n \pi_i \A_i$ is a nonsingular M-matrix. Using the definition of $\A_i$ from \eq{A}, the matrix $\B$ can be written as
\be
 \B = 
 \begin{bmatrix} 
  \I & -(\I+\C)\p \\
  -\p^\top  & \frac{1}{j}   
 \end{bmatrix}.
 \label{eq-B}
\ee
A necessary and sufficient condition for a Z-matrix to be a nonsingular M-matrix is that all its leading principal minors are positive \citep[p.\,135, condition E$_{17}$]{Berman94}. Consequently, since the $(1,1)$ block of $\B$ is the identity matrix, the matrix $\B$ will be a nonsingular M-matrix if and only if $\det (\B) > 0$. 
Using the Schur determinant formula for block matrices \citep[exercise 6.1]{Berman94}, the determinant of $\B$ can be expressed as
\be
  \det (\B) = \frac{1}{j}  - \p^\top (\I+\C) \p,
  \label{eq-det}
\ee
and hence $\det (\B) > 0$ implies $\p^\top (\I+\C) \p < \frac{1}{j}$. Moreover, for any graph $\G$ with adjacency matrix $\C$, the following is true \citep{Motzkin65}:
\be
  \frac{1}{\alpha(\G)} = \min_{\y \in \simplex} \, \y^\top (\I + \C) \y,
\label{eq-motzkin}
\ee
where $\alpha(\G)$ is the size of the maximum independent set of $\G$.
Using \eq{det} and \eq{motzkin}, we conclude that the existence of a vector $\p \in \simplex$ that satisfies 
$\det (\B) > 0$
implies 
$\frac{1}{\alpha(\G)} < \frac{1}{j}$, or $\alpha(\G) > j$, and hence it implies the existence of an independent set $V' \subseteq V$ in $\G$ with size $|V'| \geq j$. 

Since all steps in the above reduction take time at most polynomial in the size of the problem input, we conclude that the existence of a polynomial-time algorithm for deciding \mmatrix would also allow us to decide \stableset in polynomial-time, implying P=NP.
\end{proof}

\paragraph{Remarks:} 
Effectively the proof hinges on the fact that nonsingular M-matrices are closed under Schur complementation \citep[p.\,128, problem 15]{Horn91}, a property that they share with the class of positive definite matrices. However, unlike positive definite matrices, M-matrices are not closed under regular matrix addition, accounting for the difficulty of optimization problems involving M-matrices. If the input matrices are symmetric, the problem is convex, since symmetry of a nonsingular M-matrix implies its positive definiteness \citep[p.\,141, exercise 2.6]{Berman94}, in which case the corresponding optimization problem can be solved by semidefinite programming \citep{Boyd04}.


\section{Corollaries}

\cite{Berman94} provide a list of more than 50 equivalent conditions for a Z-matrix to be a nonsingular M-matrix. Our result establishes NP-hardness of all `polytopic' versions of them.
For instance, using conditions D$_{16}$ and N$_{38}$ of \citet[p.\,135,137]{Berman94}, we directly establish NP-hardness of the problems of testing if there is a convex combination of a given set of matrices, whose real eigenvalues are all positive (D$_{16}$), or whose inverse is a nonnegative matrix (N$_{38}$).

Below we briefly describe three problems whose NP-hardness follows from theorem \ref{np}, and which are of special interest for applications.

\begin{defn}[Polytopic instability]
Given a set of matrices, is there a convex combination of those that is Hurwitz stable?
\end{defn}
If $\B$ is a nonsingular M-matrix, then $-\B$ is Hurwitz. Therefore, our result establishes NP-hardness of the problem of testing the {\em instability} of a linear time-invariant system with polytopic uncertainty, providing evidence for the difficulty of obtaining low-order necessary conditions for this problem \citep{Chesi10}. The problem of testing the {\em stability} of a polytopic uncertain system is also NP-hard  \citep{Gurvits09}. We refer to \cite{Blondel97} for other complexity results in control theory.

\begin{defn}[Positive determinant]
Given a set of symmetric matrices, is there a convex combination of those whose determinant is positive?
\end{defn}
NP-hardness of this problem for non-symmetric matrices is established in the course of the proof of theorem \ref{np}. NP-hardness of the case of symmetric matrices follows from a result by \citet[theorem 5]{Grenet11} that shows that, for any given $n \times n$ matrix $\X = [x_{ij}]$, one can construct a symmetric matrix $\Y$ of dimension $O(n^3)$ with entries in $\{x_{ij}\} \cup \{0;1;-1;1/2\}$ such that $\det(\X)=\det(\Y)$.  Note that the problem becomes convex when we constrain the feasible region to positive definite matrices \citep[p.\,73]{Boyd04}.

\begin{defn}[Minimize spectral radius]
Given a set of nonnegative matrices, is there a convex combination of those whose spectral radius is less than one? 
\end{defn}
NP-hardness of this problem follows from theorem \ref{np}. Note that the matrix $\B$ in \eq{B} can be written as $\B = \I - \M_\p$, where the matrix $\M_\p = \sum_{i=1}^n \pi_i \N_i$ is a convex combination of nonnegative matrices
 \be
 \N_i = 
 \begin{bmatrix} 
  \O & \e_i+\c_i \\
  \e_i^\top  & 1-\frac{1}{j}   
 \end{bmatrix}.
 \label{eq-N}
\ee
From \citet[p.\,133, definition 1.2]{Berman94} it then follows that $\B$ is a nonsingular M-matrix if and only if the spectral radius of $\M_\p$ is strictly less than one.

This result establishes that the problem of globally minimizing spectral radius over a matrix polytope is NP-hard. The spectral radius of a nonsymmetric real matrix affine function such as $\M_\p$ is generally a nonconvex function of the parameter vector $\p$, and the problem of globally optimizing the spectral radius is a difficult one \citep{Overton88,Han99,Neumann07}. Our result provides evidence for this `difficulty'. NP-hardness of spectral radius minimization was also recently established for a class of nonnegative matrix affine functions \citep{Fercoq11}. 

We note that the problem becomes convex, and often admits tractable exact global optimization algorithms, in the case of symmetric matrices \citep{Overton88,Fiedler92}, or in the case of irreducible nonnegative matrices whose entries are posynomial functions of a parameter vector $\p$. In the latter case, minimizing spectral radius under posynomial constraints in $\p$ can be solved in polynomial time via the Collatz-Wielandt formula \citep{Friedland81} and geometric programming \citep[p.\,165]{Boyd04}. Note that in our \mmatrix problem the simplex equality constraint $\sum_{i=1}^n \pi_i = 1$ is not a posynomial constraint (it would be if the equality were replaced by the inequality `$\leq$').

  
\subsection*{Acknowledgments}

Many thanks to Michael Tsatsomeros and Leonid Gurvits for their feedback.

\bibliographystyle{apalike}

\end{document}